\documentclass[a4paper,10pt]{article}
\usepackage[latin1]{inputenc}
\usepackage{amsthm,amsmath}
\usepackage{enumerate,enumitem}
\usepackage{graphicx}
\usepackage{color}
\usepackage{cite}
\usepackage{authblk}

\theoremstyle{plain}
\newtheorem{thm}{Theorem}[section]

\newtheorem{prop}[thm]{Proposition}
\newtheorem{lem}[thm]{Lemma}
\newtheorem{defn}[thm]{Definition}
\newtheorem*{conjecture*}{Conjecture}

\newcommand{\inner}[1]{{#1}^{\circ}}
\renewcommand{\outer}[1]{\partial{#1}}

\newtheorem{claim}{Claim}

\DeclareGraphicsExtensions{.pdf,.png,.eps}
\graphicspath{{pics/}}

\title{Decomposing $4$-connected planar triangulations into two trees and one path}
\author[1]{Kolja Knauer\thanks{K.K was partially supported by ANR GATO ANR-16-CE40-0009-01.}} 
\author[2]{Torsten Ueckerdt}
\affil[1]{Aix Marseille Univ, Universit\'e de Toulon, CNRS, LIS, Marseille, France}
\affil[2]{Karlsruhe Institute of Technology (KIT), 
Institute of Theoretical Informatics}
\begin{document}

\maketitle

\begin{abstract}
 Refining a classical proof of Whitney, we show that any $4$-connected planar triangulation can be decomposed into a Hamiltonian path and two trees.
 Therefore, every $4$-connected planar graph decomposes into three forests, one having maximum degree at most~$2$.
 We use this result to show that any Hamiltonian planar triangulation can be decomposed into two trees and one spanning tree of maximum degree at most~$3$.
 These decompositions improve the result of Gon\c{c}alves [Covering planar graphs with forests, one having bounded maximum degree. \textit{J.\ Comb.\ Theory, Ser.\ B}, 100(6):729--739, 2010] that every planar graph can be decomposed into three forests, one of maximum degree at most~$4$.
 We also show that our results are best-possible.
\end{abstract}

\section{Introduction}

All graphs considered here are finite, undirected, and simple, i.e., contain no loops nor multiple edges.
The \emph{fractional arboricity} $a_f(G):=\max_{S\subseteq V; |S|\geq 2}\frac{|E(S)|}{|S|-1}$ of a graph $G = (V,E)$ (sometimes also denoted by $\Upsilon_f(G)$) is a classical measure of density, introduced by Payan~\cite{Pay-86}.
The famous Nash-Williams Theorem~\cite{Nas-61} says that the edges of any graph $G$ can be decomposed into $\lceil a_f(G)\rceil$ forests.
While the number of forests cannot be reduced, it might still be possible to improve the decomposition by imposing a low maximum degree on one forest.

For positive integers $k,d$, a graph is called \emph{$(k,d)^\star$-decomposable} if its edges decompose into $k+1$ forests, one of maximum degree at most $d$.
A weaker notion than $(k,d)^\star$-decomposable is the following:
A graph is called \emph{$(k,d)$-decomposable} if its edges decompose into $k$ forests and one subgraph of maximum degree at most $d$.
Let us also define the following notion that is stronger than $(k,d)^\star$-decomposability:
A graph is called \emph{$[k,d]^\star$-decomposable} if its edges decompose into $k$ forests and one tree of maximum degree at most $d$.

Affirming a conjecture of Balogh~\textit{et al.}~\cite{Bal-05}, Gon{\c{c}}alves~\cite{Gon-09} proves that every planar graph is $(2,4)^\star$-decomposable, while large enough planar $3$-trees are not $(2,3)$-decomposable~\cite{Bal-05}.

\paragraph{Our Results.}

In this paper we consider $[k,d]^\star$-decompositions of maximally planar graphs, which we will call planar triangulations, or just triangulations for short.
First, we improve the classical result of Whitney~\cite{Whi-32} that every $4$-connected triangulation is Hamiltonian.
Recall that a graph $G = (V,E)$ is $4$-connected if $|V| \geq 5$ and for any triple $S$ of vertices the graph $G - S$ is connected.
Moreover, a Hamiltonian path in $G$ is a simple path in $G$ that contains all vertices of $G$, and $G$ is Hamiltonian if it contains at least one Hamiltonian path.

\begin{thm}\label{thm:2t1p}
 Every $4$-connected planar triangulation $G$ decomposes into two trees and one Hamiltonian path. In particular, $G$ is $[2,2]^\star$-decomposable.
\end{thm}

Combining this result with recursive decompositions of non-$4$-connected triangulations, we obtain decompositions of Hamiltonian triangulations.

\begin{thm}\label{thm:hamil}
 Every Hamiltonian planar triangulation $G$ decomposes into two trees and a spanning tree of maximum degree~$3$. In particular, $G$ is $[2,3]^\star$-decomposable.
\end{thm}

Furthermore, our methods give a new proof of (a slight strengthening of) Gon\c{c}alves' result.

\begin{thm}\label{thm:planar}
 Every planar triangulation $G$ decomposes into two trees and a spanning tree of maximum degree~$4$. In particular, $G$ is $[2,4]^\star$-decomposable.
\end{thm}

Finally, we show that all our results are best-possible, where in the last case we provide a family of examples that is richer than just planar $3$-trees, as given in~\cite{Bal-05}.

\begin{thm}\label{thm:best-possible}
 Each of the following holds.
 \begin{enumerate}[label = (\roman*)]
  \item Some $4$-connected planar triangulations are not $(2,1)$-decomposable.\label{enum:4-connected-not-2,1}
  
  \item Some Hamiltonian planar triangulations are not $(2,2)$-decomposable.\label{enum:Hamiltonian-not-2,2}
  
  \item Some planar triangulations are not $(2,3)$-decomposable.\label{enum:planar-not-2,3}
 \end{enumerate}
\end{thm}

\paragraph{Related Work.}

While the present paper is focused on planar triangulations and the cases $k=2$, $d \in \{2,3\}$, let us give an account of the history of $(k,d)^\star$-decomposability and $(k,d)$-decomposability for general graphs and general integers $k,d \geq 1$.

One motivation for $(k,d)^\star$-decomposability are applications to bounding the (incidence) game-chromatic number~\cite{He-02,Mon-10,Cha-15} and the spectral radius of a graph~\cite{Dvo-10}.
However, most of the research in this field was inspired by the famous Nine Dragon Tree Conjecture of Montassier, Ossona de Mendez, Raspaud, and Zhu~\cite{Mon-12}, which states that if the difference between $\lceil a_f(G)\rceil$ and $a_f(G)$ is large, then $G$ decomposes into $\lceil a_f(G) \rceil$ forests, where the maximum degree of one forest can be bounded.
More precisely, if $a_f(G)\leq k+\frac{d}{k+d+1}$ for positive integers $k,d$, then $G$ is $(k,d)^\star$-decomposable.
The Nine Dragon Tree Conjecture was proved for several special cases~\cite{Kai-11,Mon-12,Kim-13,Che-17} before it was confirmed in full generality by Jiang and Wang in 2016~\cite{Jia-16}.

The original motivation for the Nine Dragon Tree Conjecture in~\cite{Mon-12} comes as a generalization of decomposition results in sparse planar graphs: 
Planar graphs are known to be $(1,1)^\star$-decomposable, i.e., decompose into a forest and a matching, when they have girth 
at least $8$~\cite{Wan-11,Mon-12}, while some planar graphs of girth~$7$ are not $(1,1)$-decomposable~\cite{Mon-12}.
For $d=2$, He~\textit{et al.}~\cite{He-02} show that planar graphs of girth at least~$7$ are $(1,2)$-decomposable, which was improved to $(1,2)^\star$-decomposability by Gon\c{c}alves~\cite{Gon-09}, while some planar graphs of girth~$5$ are not $(1,2)$-decomposable~\cite{Mon-12}.
In~\cite{He-02} it is further shown that planar graphs of girth at least~$5$ are $(1,4)$-decomposable.
All these decomposition results follow immediately from the Nine Dragon Tree Theorem~\cite{Jia-16} as for every planar graph $G$ of girth at least~$g$ we have $a_f(G) \leq \frac{g}{g-2}$, and thus these decompositions rely purely on the low fractional arboricity of graphs.

However, for planar triangulations the fractional arboricity tends to $3$ as the number of vertices tends to infinity.
Thus the Nine Dragon Tree Theorem does not give $(2,d)^\star$-decomposability of all planar graphs for any fixed $d$.
Hence, for the following results (as well as our results in the present paper) the structure of planar graphs had to be exploited on a different level:
In~\cite{He-02} it is shown that planar graphs are $(2,8)$-decomposable, which is strengthened to $(2,8)^\star$-decomposability in~\cite{Bal-05}.
Moreover, Balogh~\textit{et al.}~\cite{Bal-05} show that Hamiltonian and consequently $4$-connected planar graphs are $(2,6)$-decomposable.
Finally Gon\c{c}alves~\cite{Gon-09} improved these results to $(2,4)^\star$-decomposability of all planar graphs, which is best-possible~\cite{Bal-05}.

\paragraph{Organization of the Paper.}

In Section~\ref{sec:Whitney-lemma} we prove our key lemma, Lemma~\ref{lem:Whitney}, which is crucial for all our decomposition results. 
In Section~\ref{sec:triangle} we decompose any planar triangulation along its separating triangles and introduce triangle assignments.
Here we also combine Lemma~\ref{lem:Whitney} to obtain $[2,k+2]^\star$-decompositions for $k \in \{0,1,2\}$ of planar triangulations admitting so-called $k$-assignments, c.f.\ Proposition~\ref{prop:decomposition-from-assignment}.
In Section~\ref{sec:proofs} we prove our main decomposition results, namely Theorems~\ref{thm:2t1p}--\ref{thm:planar}.
To this end, we show that $4$-connected (respectively Hamiltonian and general) triangulations admit $0$-assignments (respectively $1$-assignments and $2$-assignments) and use the decompositions given by Proposition~\ref{prop:decomposition-from-assignment} in Section~\ref{sec:triangle}.
In Section~\ref{sec:tight} we show that our results are best-possible by constructing $4$-connected (respectively Hamiltonian and general) triangulations that are not $(2,1)$-decomposable (respectively $(2,2)$-decomposable and $(2,3)$-decomposable); In other words, we prove Theorem~\ref{thm:best-possible}.
Finally, we conclude the paper in Section~\ref{sec:conc}.

\section{The Key Lemma}\label{sec:Whitney-lemma}

This section is devoted to the proof of Lemma~\ref{lem:Whitney}, which is a central element of the proofs of all our Theorems.

Let $G$ be a plane embedded graph with a simple outer cycle $C$. Moreover, let $G$ be inner triangulated, that is, every inner face of $G$ is a triangle. For two outer vertices $u,v$ of $G$ we denote by $P_{uv}$ the path from $u$ to $v$ along the outer cycle in counterclockwise direction, and define $\inner{P}_{uv} = P_{uv} \setminus \{u,v\}$. If $u=v$ then $P_{uv}$ consists of only one vertex and $\inner{P}_{uv} = \emptyset$. A \emph{filled triangle} in $G$ is a triple of pairwise adjacent vertices, such that at least one vertex of $G$ lies inside this triangle.
A \emph{separating triangle} in $G$ is a triple of pairwise adjacent vertices, such that at least one vertex of $G$ lies inside this triangle and at least one vertex lies outside this triangle. It is well-known and easy to see that a planar triangulation is $4$-connected if and only if it does not have\footnote{Using planar triangulation instead of plane triangulation is justified as the existence of separating triangles is independent of the chosen plane embedding.} any separating triangles.

\begin{defn}
 A plane inner triangulated graph $G = (V,E)$ with simple outer cycle $C$ is a \emph{Whitney graph with respect to $(x,y,z)$} if $x,y,z$ are outer vertices of $G$ with $z \neq x,y$, such that:
 \begin{itemize}
  \item $G$ contains no filled triangle.
  \item $x,y,z$ appear in this counterclockwise order around $C$.
  \item $P_{xy}$, $P_{yz}$, $P_{zx}$ are induced paths in $G$.
  \item If $x=y$, then $zx$ is not an edge of $G$.
 \end{itemize}
\end{defn}

Note that the outer face of a Whitney graph with at least four vertices cannot be a triangle.
Moreover, any inner triangulated $4$-connected graph with some outer vertices $x,y,z$ in this counterclockwise order is a Whitney graph with respect to $(x,y,z)$. 

\begin{lem}\label{lem:Whitney}
 If $G = (V,E)$ is a Whitney graph with respect to $(x,y,z)$, then the edges of $G$ can be oriented and colored black, red, and blue, such that each of the following holds.
 \begin{enumerate}[label = (\arabic*)]
  \itemsep0pt
  \item The black edges form a directed Hamiltonian path from $x$ to $z$ in $G$.\label{enum:hamil-path}

  \item Every inner vertex has precisely one outgoing red edge and one outgoing blue edge.\label{enum:inner-two}

  \item Every vertex on $\inner{P}_{xy}$ has precisely one outgoing blue edge and no outgoing red edge.\label{enum:xy-blue}

  \item Every vertex on $\inner{P}_{yz}$ has precisely one outgoing red edge and no outgoing blue edge.\label{enum:yz-red}

  \item Every vertex on $\inner{P}_{zx}$ has precisely one outgoing blue edge and no outgoing red edge.\label{enum:zx-blue}

  \item Neither $y$ nor $z$ has outgoing red edges nor outgoing blue edges.\label{enum:y-z-no}

  \item If $x \neq y$, then vertex $x$ has precisely one outgoing blue edge which is $xz$ if it exists and no outgoing red edge.
   If $x=y$, then $x$ has no outgoing blue and no outgoing red edge.\label{enum:x}

  \item There is no monochromatic directed cycle in $G$.\label{enum:acyclic}
 \end{enumerate}
\end{lem}
\begin{proof}
 Recall that our proof is based on the decomposition of planar $4$-connected triangulations by Whitney~\cite{Whi-32}. We do induction on the number $|V|$ of vertices of $G$. 
 If $|V|=3$ then $G$ is a triangle with vertices $x,y$ and $z$.
 We define $P$ to be the path $x-y-z$ and orient $xz$ from $x$ to $z$ and color it blue.
 It is easy to see that~\ref{enum:hamil-path}--\ref{enum:acyclic} are satisfied.
 If $|V| > 3$ we distinguish seven cases, which we go through in this order, i.e., when considering Case~$i$ we sometimes make use of the fact that Case~$j$ does not apply for $j < i$.
 
 \begin{description}
  \item[Case~1: $x=y$.]{\ \\}
   Let $x'$ and $y'$ be the neighbor of $x$ on $P_{zx}$ and $P_{yz}$, respectively.
   Then $G' = G \setminus \{x\}$ is a Whitney graph with respect to $(x',y',z)$.
   Indeed the vertices on $P_{x'y'}$ are neighbors of $x$ and hence a chord $\{u,v\}$ in $P_{x'y'}$ would give a separating triangle $\{x,u,v\}$ in $G$.
   Hence $P_{x'y'}$ is an induced path in $G'$.
   Moreover, $P_{y'z}$ and $P_{z,x'}$ are subsets of the induced paths $P_{yz}$ and $P_{zx}$, respectively, and thus induced, too.
 
   By induction there is a Hamiltonian path $P'$ from $x'$ to $z$ in $G'$ and an orientation and coloring of the edges in $E(G') \setminus E(P)$, such that~\ref{enum:inner-two}--\ref{enum:acyclic} are satisfied. We extend $P'$ by the edge $xx'$, i.e., $P = \{xx'\} \cup P'$, and the coloring/orientation by orienting all incident edges at $x$ (except $xx'$) towards $x$ and coloring them red. See Figure~\ref{fig:case1} for an illustration.
 
   \begin{figure}[tb]
    \centering
    \includegraphics{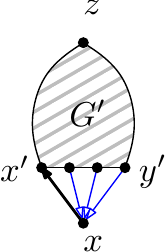}
    \caption{\textbf{Case~1:} $x=y$}
    \label{fig:case1}
   \end{figure}

   We need to argue that $P$ together with the coloring/orientation satisfies~\ref{enum:inner-two}--\ref{enum:acyclic}. Indeed~\ref{enum:inner-two} follows from~\ref{enum:inner-two},~\ref{enum:xy-blue} with respect to $G'$ and the coloring/orientation of the edges incident to $x$, for~\ref{enum:xy-blue} there is nothing to show,~\ref{enum:yz-red}--\ref{enum:y-z-no} follow from~\ref{enum:yz-red}--\ref{enum:y-z-no} with respect to $G'$, and~\ref{enum:x} follows again from the coloring/orientation of the edges incident to $x$. Finally, we need to show~\ref{enum:acyclic}, namely that there is no directed monochromatic cycle in $G$. By~\ref{enum:acyclic} with respect to $G'$ such a cycle would contain $x$, which has not incoming and outgoing edges of the same color.
  
  \item[Case~2: There is an edge $xu$ with $u \in \inner{P}_{yz}$.]{\ \\}  
   We choose $u \in \inner{P}_{yz}$ to be the vertex that is a neighbor of $x$ and is closest to $z$ on $\inner{P}_{yz}$. The three illustrations in Figure~\ref{fig:case2} display the three sub-cases that we sometimes have to treat differently along the construction: $zx\notin E$,  $zx\in E$ and $uz\notin E$, and $zx,uz\in E$.
 
   Define $G_1$ to be the inner triangulated subgraph of $G$ with outer cycle $P_{xy} \cup P_{yu} \cup \{ux\}$. Then $G_1$ is a Whitney graph with respect to $(x,y,u)$. We define $G_2$ to be the graph $G \setminus (G_1 \setminus \{u\})$. If both, $xz$ and $uz$, are edges in $G$, then $G_2$ is just a single edge, which we put into the Hamiltonian path. Otherwise $G_2$ is a Whitney graph with respect to $(u,y',z)$ (or rather $(u,u,z)$ in case $xz \in E$), where $y'$ is the neighbor of $x$ on $P_{zx}$ if $xz \notin E$. However, if $xz \notin E$ the embedding of $G_2$ needs to be flipped so that $u,y',z$ appear in counterclockwise order. Indeed $P_{y'u}$ is induced since it consists solely of neighbors of $x$, and similarly $P_{uz}$ is induced in case $xz \in E$. We apply induction to both, $G_1$ and $G_2$ (if $G_2$ is not just an edge), and concatenate the obtained Hamiltonian paths in $G_1$ and $G_2$ to a Hamiltonian path from $x$ to $z$ in $G$.
 
   It remains to color and orient the edges incident to $x$, but not contained in $G_1$. Additionally, we also define a color and orientation for the edge $ux$, disregarding its color and orientation given by induction on $G_1$. (Note that $ux$ is certainly not in the Hamiltonian path of $G_1$.) If $xz \in E$ we orient these edges, as well as the edge $ux$, towards $x$ and color them red, except for $xz$ which we color blue and orient to $z$. On the other hand if $xz \notin E$ we swap the colors red and blue in the coloring for $G_2$, but keep the orientation the same. Moreover, we color the edges incident to $x$ blue and orient them towards $x$, except for $xu$ which is oriented towards $u$. See Figure~\ref{fig:case2} for an illustration.
 
   \begin{figure}[tb]
    \centering
    \includegraphics{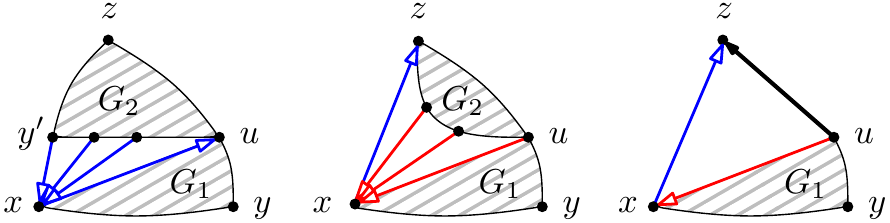}
    \caption{Considered subcases of \textbf{Case~2:} There is an edge $xu$ with $u \in \inner{P}_{yz}$.}
    \label{fig:case2}
   \end{figure}
 
   It is straightforward to check that~\ref{enum:inner-two}--\ref{enum:x} follow from~\ref{enum:inner-two}--\ref{enum:x} with respect to $G_1$ and $G_2$ and the coloring/orientation of the edges incident to $x$. Moreover, by~\ref{enum:acyclic} with respect to $G_1$ and $G_2$, every directed monochromatic cycle has to contain $x$. The only case in which $x$ has incoming and outgoing edges of the same color is when $xz \notin E$. There the only cycle would have to be blue and go through $u$ which has no blue outgoing edges by~\ref{enum:x} for $G_2$ and since red and blue were swapped in $G_2$. Thus there is no such directed monochromatic cycle in $G$, i.e.,~\ref{enum:acyclic} is satisfied.
   
  \item[Case~3: There is an edge $uv$ with $u \in \inner{P}_{xy}$ and $v \in P_{zx}$.]{\ \\}  
   We choose $u \in \inner{P}_{xy}$ to be the vertex that has a neighbor on $P_{zx}$ and is closest to $x$ on $P_{xy}$, and $v$ to be the neighbor of $u$ on $P_{zx}$ that is closest to $z$ on $P_{zx}$. Note that $u,v \neq x,y$, but possibly $v=z$. The three illustrations in Figure~\ref{fig:case2} display the three sub-cases that we sometimes have to treat differently along the construction: $xv\notin E$,  $xv\in E$ and $xu\notin E$, and $xv,xu\in E$. We define $G_2$ to be the inner triangulated subgraph of $G$ with outer cycle $P_{uy} \cup P_{yz} \cup P_{zv} \cup \{vu\}$. Then $G_2$ is a Whitney graph with respect to $(u,y,z)$. Indeed, $P_{zu}$ in $G_2$ is induced since $P_{zx}$ in $G$ is and by the choice of $v$ there is no edge between $u$ and $P_{zv} \setminus \{v\}$. Next consider $G_1 = G \setminus (G_2 \setminus \{u\})$ and the neighbor $y'$ of $v$ on $P_{vx}$. If both, $xu$ and $xv$, are edges in $G$, then $G_1$ is just the edge $xu$ and we put this edge into the Hamiltonian path. Otherwise $G_1$ is a Whitney graph with respect to $(x,y',u)$, since, by the choice of $u$, there is no edge from $v$ to a vertex on $P_{xy}$ between $x$ and $u$. However, the embedding of $G_1$ needs to be flipped so that $x,y',u$ appear in counterclockwise order. 
 
   We apply induction to both, $G_1$ (if $G_1$ is not just an edge) and $G_2$, and concatenate the obtained Hamiltonian paths in $G_1$ and $G_2$ to a Hamiltonian path from $x$ to $z$ in $G$. We orient the edges incident to $v$ towards $v$ and color them blue. Note that if $y'=x$ then the vertices on $P_{xy}$ that are contained in $G_1$ have no outgoing \emph{red} edge, since we flipped the embedding of $G_1$. Similarly the neighbors of $v$ in $G_1$ have no outgoing blue edge within $G_1$.
 
   \begin{figure}[tb]
    \centering
    \includegraphics{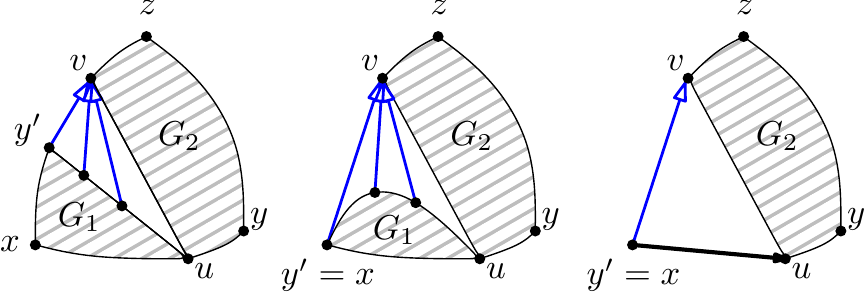}
    \caption{Considered subcases of \textbf{Case~3:} There is an edge $uv$ with $u \in \inner{P}_{xy}$ and $v \in P_{zx}$.}
    \label{fig:case3}
   \end{figure}
 
   It is again straightforward to check that~\ref{enum:inner-two}--\ref{enum:x} follow from~\ref{enum:inner-two}--\ref{enum:x} with respect to $G_1$ and $G_2$ and the coloring/orientation of the edges incident to $v$ and $u$. It remains to show that~\ref{enum:acyclic} is satisfied, i.e., there is no monochromatic directed cycle in $G$. By~\ref{enum:acyclic} with respect to $G_1$ and $G_2$, such a cycle would contain edges from $G_1$ to $v$ and pass through $u$, but $u$ has no outgoing edges towards $G_1$.
   
  \item[Case~4: $xz \in E$.]{\ \\}  
   Let $x'$ be the neighbor of $x$ in $P_{xy}$. Note that possibly $x' = y$, which is illustrated second in Figure~\ref{fig:case4}. Moreover $x'z$ is not an edge since Case~3 does not apply. Now $G' = G \setminus \{x\}$ is a Whitney graph with respect to $(x',y,z)$, since the vertices in $P_{zx'}$ are neighbors of $x$ and hence $P_{zx'}$ is induced. We apply induction to $G'$ and extend the obtained Hamiltonian path $P'$ by the edge $xx'$ and orient $xz$ from $x$ to $z$ and color it blue. We orient the remaining edges incident to $x$ towards $x$ and color them red. See Figure~\ref{fig:case4} for an illustration.
 
   \begin{figure}[tb]
    \centering
    \includegraphics{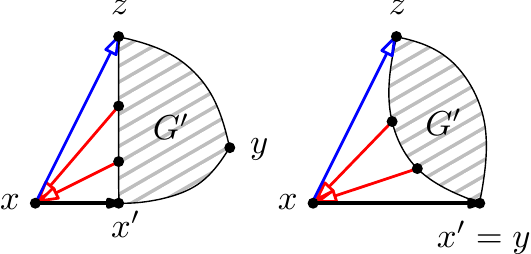}
    \caption{Considered subcases of \textbf{Case~4:} $xz \in E$}
    \label{fig:case4}
   \end{figure}
 
   Now~\ref{enum:inner-two} follows from~\ref{enum:inner-two},~\ref{enum:zx-blue} with respect to $G'$ and the coloring/orientation of the edges incident to $x$,~\ref{enum:xy-blue},~\ref{enum:yz-red} and~\ref{enum:y-z-no} follow from~\ref{enum:xy-blue},~\ref{enum:yz-red} and~\ref{enum:y-z-no} (and~\ref{enum:x} in case $x'=y$) with respect to $G'$, and~\ref{enum:zx-blue} and~\ref{enum:x} follow again from the orientation/coloring of the edges at $x$. Finally,~\ref{enum:acyclic} follows from~\ref{enum:acyclic} with respect to $G'$ and the fact that $x$ has no outgoing and incoming edges of the same color.
   
  \item[Case~5: $yz \in E$.]{\ \\}  
   Let $y'$ be the neighbor of $z$ in $P_{zx}$. Note that $y' \neq x$ since Case~4 does not apply. Now $G' = G \setminus \{z\}$ is a Whitney graph with respect to $(x,y',y)$. However, the embedding of $G'$ needs to be flipped so that $x,y',y$ appear in counterclockwise order. Indeed $P_{yy'}$ is an induced path since all its vertices are neighbors of $z$. We apply induction to $G'$ and extend the obtained Hamiltonian path $P'$ by the  edge $yz$. We orient the remaining edges incident to $z$ towards $z$ and color them blue. 
   See Figure~\ref{fig:case5} for an illustration.
 
   \begin{figure}[tb]
    \centering
    \includegraphics{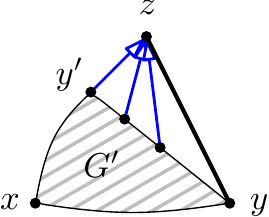}
    \caption{\textbf{Case~5:} $yz \in E$}
    \label{fig:case5}
   \end{figure}
 
   Now~\ref{enum:inner-two} follows from~\ref{enum:inner-two},~\ref{enum:yz-red} with respect to $G'$ and the coloring/orientation of the edges incident to $z$,~\ref{enum:xy-blue} and~\ref{enum:zx-blue}--\ref{enum:x} follow from~\ref{enum:xy-blue} and~\ref{enum:zx-blue}--\ref{enum:x} with respect to $G'$, and for~\ref{enum:yz-red} there is nothing to show. Finally,~\ref{enum:acyclic} follows from~\ref{enum:acyclic} with respect to $G'$ and the fact that $z$ has no outgoing and incoming edges of the same color.
   
  \item[Case~6: There is an edge $yu$ with $u \in \inner{P}_{zx}$.]{\ \\}  
   We choose $u$ to be any neighbor of $y$ on $P_{zx}$. Note that $u \neq z$ since Case~5 does not apply. We define $G_2$ to be the inner triangulated subgraph of $G$ with outer cycle $P_{yz} \cup P_{zu} \cup \{uy\}$. Then $G_2$ is a Whitney graph with respect to $(y,u,z)$. However, the embedding of $G_2$ needs to be flipped so that $y,u,z$ appear in counterclockwise order. Indeed, $P_{zy}$ in $G_2$ is induced since $P_{zu}$ in $G$ is so and by the choice of $u$ there is no edge between $y$ and $P_{zu} \setminus \{u\}$. Next consider $G_1 = G \setminus (G_2 \setminus \{y\})$ and the neighbor $y'$ of $u$ on $P_{ux}$. The three illustrations in Figure~\ref{fig:case6} display the three sub-cases that we sometimes have to treat differently along the construction: $xu\notin E$,  $xu\in E$ and $xy\notin E$, and $xu,xy\in E$. If both, $ux$ and $xy$, are edges in $G$, then $G_1$ is just the edge $xy$ and we put this edge into the Hamiltonian path. Otherwise $G_1$ is a Whitney graph with respect to $(x,y',y)$. However, the embedding of $G_1$ needs to be flipped so that $x,y',y$ appear in counterclockwise order.
 
   We apply induction to both, $G_1$ (if $G_1$ is not just an edge) and $G_2$, and concatenate the obtained Hamiltonian paths in $G_1$ and $G_2$ to a Hamiltonian path from $x$ to $z$ in $G$. We swap the colors red and blue in the coloring for $G_2$, but keep the orientation the same. We orient the edges incident to $u$ towards $u$ and color them blue. Note that if $y'=x$ then the vertices on $P_{xy}$ that are contained in $G_1$ have no outgoing \emph{red} edge, since we flipped the embedding of $G_1$. Similarly the neighbors of $u$ in $G_1$ have no outgoing blue edge within $G_1$.

   \begin{figure}[tb]
    \centering
    \includegraphics{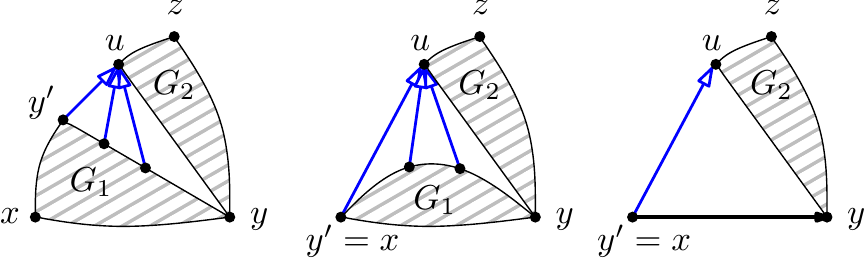}
    \caption{Considered subcases of \textbf{Case~6:} There is an edge $yu$ with $u \in \inner{P}_{zx}$.}
    \label{fig:case6}
   \end{figure}
 
   It is again straightforward to check that~\ref{enum:inner-two}--\ref{enum:x} follow from~\ref{enum:inner-two}--\ref{enum:x} with respect to $G_1$ and $G_2$ and the coloring/orientation of the edges incident to $u$ and $y$. By~\ref{enum:y-z-no} there is no outgoing edge at $y$ towards $G_1$. Thus no directed cycle contains $y$ and hence from~\ref{enum:acyclic} with respect to $G_1$ and $G_2$ follows~\ref{enum:acyclic} for $G$, i.e., there is no directed monochromatic cycle in $G$.
   
  \item[Case~7: None of Case~1 -- Case~6 applies.]{\ \\}
   Let $uv$ be the edge between a vertex $u\in\inner{P}_{xy}$ and a vertex $v\in\inner{P}_{yz}$ farthest away from $y$. If no such edge exists set $u=y$ and let $v$ be the neighbor of $y$ and on $P_{yz}$. Then $u \neq x$ and $v \neq z$, since otherwise an earlier case would apply.
 
   Let $w$ be the neighbor of $x$ on $P_{zx}$. Since earlier cases do not apply, we have $w \neq z$. Consider the subgraph $\tilde{G}$ of $G$ induced by all vertices that are not on $P_{xy}$ but have at least one neighbor in $P_{xy}$. Since none of Case~2, Case~3 and Case~6 applies, $v$ and $w$ are the only outer vertices of $G$ that are contained in $\tilde{G}$. Let $P$ be a shortest $vw$-path in $\tilde{G}$. (Clearly $\tilde{G}$ is connected since $G$ is inner triangulated.) We define $G'$ to be the inner triangulated subgraph of $G$ with outer cycle $P \cup P_{vz} \cup P_{zw}$. Then $G'$ is a Whitney graph with respect to $(v,w,z)$. However, the embedding of $G'$ needs to be flipped so that $v,w,z$ appear in counterclockwise order. Indeed, the path $P$ is induced since it is a shortest path.

   Moreover, if $u\neq y$ define $G''$ to be the inner triangulated subgraph of $G$ with outer cycle $P_{uy} \cup P_{yv} \cup vu$, which clearly is a Whitney graph with respect to $(u,y,v)$. If $u=y$, then $v$ is the neighbor of $y$ on $P_{yz}$ and we set $G''$ to be the edge $uv$.
 
   Furthermore we define $G_1,\ldots,G_k$ to be the blocks of $G \setminus (G'\cup G''\setminus\{u\})$, where the numbering is according to their appearance along $P_{xy}$. For every $i=1,\ldots,k$ the graph $G_i$ contains two vertices $x_i$ and $z_i$ that lie on $P_{xy}$ and have a neighbor on $P$. Let $x_i$ be the one that is closer to $x$ on $P_{xy}$. If $x_i$ and $z_i$ have a common neighbor on $P$, then $G_i$ is either just an edge or a Whitney graph with respect to $(x_i,x_i,z_i)$. Otherwise there is a unique vertex $y_i$ in $G_i$ that has two neighbors on $P$, and $G_i$ is a Whitney graph with respect to $(x_i,y_i,z_i)$. Whenever $G_i$ is not just an edge, we flip the embedding of $G_i$ so that $(x_i,y_i,z_i)$ appear in counterclockwise order. 

   We color and orient all the Whitney graphs $G', G'', G_1, \ldots, G_k$ by induction, where if $G''=uv$, we put $uv$ directed from $u$ to $v$ into the black path, and whenever $G_i = x_iz_i$, we put $x_iz_i$ directed from $x_i$ to $z_i$ into the black path. In $G'$ we swap the colors red and blue.

   The remaining edges emanating from $P$ outside $G'$ are oriented as follows. All such edges at a vertex on $P$ are blue and oriented towards it until the last edge from $P_{xy}$, which is blue but oriented towards $P_{xy}$. If there are more edges to the vertex they are red and oriented towards it.  See Figure~\ref{fig:case7} for an illustration.
 
   \begin{figure}[tb]
    \centering
    \includegraphics{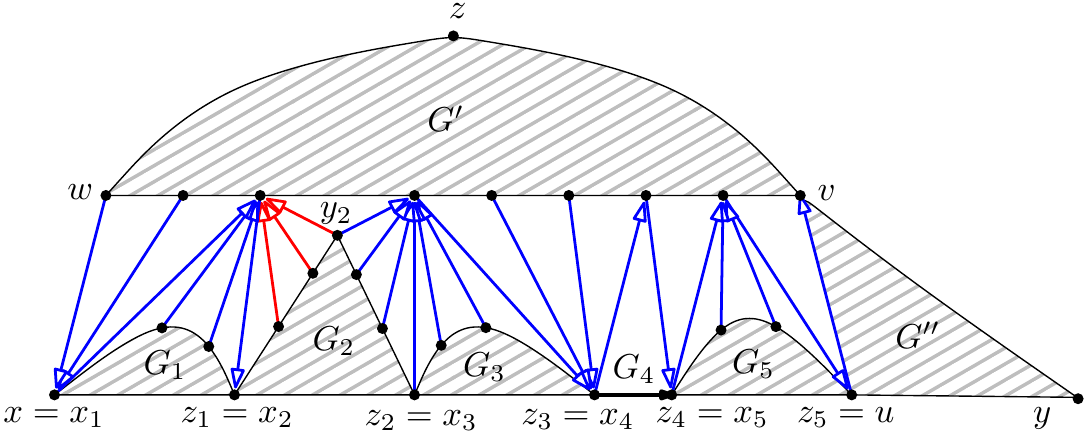}
    \caption{\textbf{Case~7:} None of Case~1 -- Case~6 applies.}
    \label{fig:case7}
   \end{figure}
 
   It is maybe a bit tedious but again straightforward to check that~\ref{enum:inner-two}--\ref{enum:x} follow from~\ref{enum:inner-two}--\ref{enum:x} with respect to $G', G'', G_1, \ldots, G_k$ after swapping red and blue in $G'$.
 
   In order to see~\ref{enum:acyclic}, first note that no red edges leave $G'$ and thus any monochromatic cycle would have to be blue.
   Moreover, vertices on $P$ have no blue outgoing edges towards $G'$, i.e., a blue cycle has to use $G_1, \ldots, G_k$ and the explicitly colored edges. The latter form a set of rooted trees, where the only vertices with ingoing blue edges in $G_1, \ldots, G_k$ are the vertices $x_1, z_1, \ldots, z_k$, since these have no outgoing blue edges within the respective $G_i$ any directed blue path continues to the right until eventually reaching $u$ and then possibly $v$, which has no blue outgoing edge. This gives~\ref{enum:acyclic} and concludes the proof.\qedhere
 \end{description} 
\end{proof}

\section{Triangle Assignments}\label{sec:triangle}

In this section we will establish the notion of triangle assignments and  under which conditions they are sufficient to apply Lemma~\ref{lem:Whitney} to recursive decompositions of planar triangulations. 

Let $G$ be a plane triangulation and let $X = X(G)$ and $Y = Y(G)$ denote the set of filled triangles in $G$ and inner faces in $G$, respectively.
In particular, $X \dot\cup Y$ is a partition of the triangles in $G$, we have $|X|=0$ if and only if $G$ is just a triangle, and we have $|X| = 1$ if and only if $G$ is $4$-connected.
The set $X$ is naturally endowed with a partial order $\prec$, where $\Delta \prec \Delta'$ whenever the interior of $\Delta$ is strictly contained in the interior of $\Delta'$.
Since any two filled triangles $\Delta_1,\Delta_2$ that both contain a third filled triangle $\Delta_3$ (i.e., $\Delta_3 \prec \Delta_1$ and $\Delta_3 \prec \Delta_2$) are necessarily contained in each other ($\Delta_1 \prec \Delta_2$ or $\Delta_2 \prec \Delta_1$), we have that the partially ordered set $(X,\prec)$ has the structure of a rooted tree $T_G$ whose root is the outer triangle $\Delta_{\rm out}$ and where $\Delta \prec \Delta'$ if and only if $\Delta,\Delta'$ lie on a root-to-leaf path in $T$ with $\Delta'$ being closer to the root $\Delta_{\rm out}$ than $\Delta$.
Tree $T_G$ is sometimes called the \emph{separation tree} of $G$, see~\cite{Sun-93} for an early appearance of the term.

For each triangle $\Delta \in X$, let $G_\Delta$ denote the subgraph of $G$ induced by $\Delta$ and all vertices inside $\Delta$.
We further define $\outer{G}$ to be the plane triangulated graph obtained from $G$ by removing for each $\Delta \in X(G) - \Delta_{\rm out}$ all inner vertices of $G_\Delta$.
Note that $\outer{G} = G$ if and only if $G$ is $4$-connected.
For the graphs $G_\Delta$ and $\outer{G_\Delta}$ with $\Delta \in X$ we have that
\begin{enumerate}[label = \textbf{(P\arabic*)}]
 \item $\outer{G_\Delta}$ is a $4$-connected triangulation for every $\Delta \in X(G)$,\label{enum:4-connected-parts}
 
 \item every separating triangle of $G$ is an inner face of $\outer{G_\Delta}$ for exactly one $\Delta \in X(G)$,\label{enum:separating-triangles-are-faces}
 
 \item the graphs $\{\outer{G_\Delta} - E(\Delta) \mid \Delta \in X(G)\}$ form an edge-partition of $G$.\label{enum:parts-form-partition}
\end{enumerate}
Properties~\ref{enum:4-connected-parts}--\ref{enum:parts-form-partition} will enable us to find $[2,k]^\star$-decompositions of $G$ for small $k$, based on $[2,2]^\star$-decompositions of $\outer{G_\Delta} - E(\Delta)$, $\Delta \in X(G)$, given by the following lemma.
Let $\Delta_{\rm out} = v_0v_1v_2$ be the outer triangle of $G$.
An outer vertex of $\outer{G} - E(\Delta_{\rm out})$ different from $v_0,v_1,v_2$ is called \emph{special vertex} of $G$, and is denoted by $u_i$, when it is adjacent to $v_{i-1}$ and $v_{i+1}$ (indices modulo~$3$). We also say that $u_i$ is the special vertex \emph{opposing} $v_i$.
Note that $u_0,u_1,u_2$ are pairwise distinct when $|V(\outer{G})| \geq 5$, and pairwise coincide when $|V(\outer{G})| = 4$, i.e., $\outer{G}\cong K_4$. An immediate consequence of Lemma~\ref{lem:Whitney} is the following. 

\begin{lem}\label{lem:inner-decomposition}
 Let $G$ be a plane triangulation with outer triangle $\Delta_{\rm out} = v_0v_1v_2$ and corresponding special vertices $u_0,u_1,u_2$.
 Then for any $\{x,y,z\}= \{0,1,2\}$ the edges of $\outer{G} - E(\Delta_{\rm out})$ can be partitioned into three forests $F_x,F_y,F_z$ such that
 \begin{itemize}
  \item $F_x$ is a Hamiltonian path of $(\outer{G}- E(\Delta_{\rm out}))\setminus\{v_y,v_z\}$ going from $v_x$ to $u_x$,
  
  \item $F_y$ is a spanning tree of $(\outer{G}- E(\Delta_{\rm out}))\setminus\{v_x,v_z\}$,
  
  \item $F_z$ is a spanning forest of $(\outer{G}- E(\Delta_{\rm out}))\setminus\{v_y\}$ consisting of two trees, one containing $v_x$ and one containing $v_z$, unless $\outer{G}\cong K_4$. In this case $F_z=v_zu_z$.
 \end{itemize}
\end{lem}
\begin{proof}
 If $\outer{G}$ has only $4$ vertices, i.e., $\outer{G}\cong K_4$, then the decomposition $F_i=v_iu_i$ for all  $i\in\{x,y,z\}$ has the desired properties. So assume that $\outer{G}$ has at least $5$ vertices.
 Since $\outer{G}$ is $4$-connected, $\outer{G} - V(\Delta_{\rm out})$ is a Whitney graph with respect to $(u_0,u_1,u_2)$. By rotating and flipping we can assume without loss of generality that $(y,z,x) = (0,1,2)$. Then by Lemma~\ref{lem:Whitney} the edges of $\outer{G}- V(\Delta_{\rm out})$ can be oriented and colored black, red and blue with the following properties. The black edges form a Hamiltonian path from $u_0$ to $u_2$, which we extend by the edge $u_0v_2=u_yv_x$ to obtain $F_x$. Every inner vertex has exactly one outgoing red edge and one outgoing blue edge, every vertex on $\inner{P}_{u_2u_0}\cup\inner{P}_{u_0u_1}\cup\{u_0\}$ has an outgoing blue but no outgoing red edge, every vertex on $\inner{P}_{u_1u_2}$ has an outgoing red but no outgoing blue edge, and neither $u_1$ nor $u_2$ have outgoing blue or red edges. Since there are no directed monochromatic cycles. This implies that red and blue edges form a forest each, where the roots of the red trees correspond to the vertices on ${P}_{u_2u_0}\cup{P}_{u_0u_1}$ and the roots of the blue trees ${P}_{u_1u_2}$. Thus, coloring all edges except $u_yv_x$ from ${P}_{u_2u_0}\cup{P}_{u_0u_1}$ to $v_x$ or $v_z$ red gives $F_z$ and coloring all edges from ${P}_{u_1u_2}$ to $v_y$ blue gives $F_y$. See Figure~\ref{fig:lemma31} for an illustration. We have obtained the desired decomposition of the edges of $\outer{G} - E(\Delta_{\rm out})$.
 \begin{figure}[tb]
  \centering
  \includegraphics{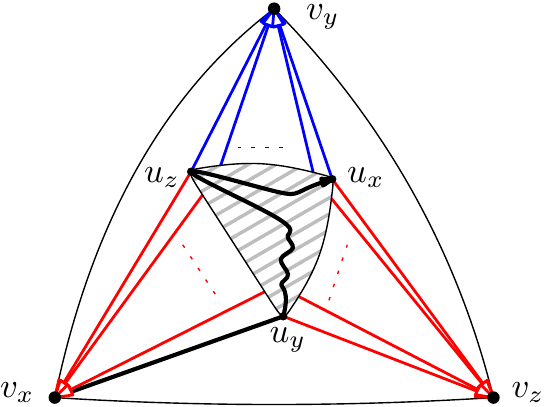}
  \caption{The construction in Lemma~\ref{lem:inner-decomposition}.}
  \label{fig:lemma31}
 \end{figure}
\end{proof}

We define for $k \in \{0,1,2\}$ a \emph{$k$-assignment} of a plane triangulation $G$ with outer triangle $\Delta_{\rm out}$ to be a map $\phi : X(G) \to V(G)$ which satisfies $\phi(\Delta) \in V(\Delta)$ for every $\Delta \in X$ and
\[
 |\phi^{-1}(v)| \leq
  \begin{cases}
   k+1 & \text{ if $v = \phi(\Delta_{\rm out})$ or there is some $\Delta \in X(G)$ such that}\\
       & \quad \text{ $v$ is the special vertex in $G_\Delta$ opposing $\phi(\Delta)$,}\\
   k & \text{ otherwise.}
  \end{cases}
\]

\begin{prop}\label{prop:decomposition-from-assignment}
 Let $G$ be a plane triangulation with outer triangle $\Delta_{\rm out} = w_0w_1w_2$ such that $w_0w_1$ or $w_0w_2$ is in no separating triangle.
 If $G$ admits a $k$-assignment $\phi$, $k \in \{0,1,2\}$, such that $\phi(\Delta_{\rm out})=w_0$, then $G$ can be decomposed into two trees and one spanning tree with maximum degree at most $k+2$, i.e., $G$ is $[2,k+2]^\star$-decomposable.
\end{prop}
\begin{proof}
 Without loss of generality assume that $w_0w_2$ is in no separating triangle.

 We will prove that $G$ decomposes into three trees $T_0, T_1, T_2$, where $T_0$ is spanning and has maximum degree $k+2$, $T_1$ spans $G\setminus\{w_0, w_2\}$, and $T_2$ spans $G\setminus\{w_1\}$.
 Even stronger, for every vertex $v$ we shall have $\deg_{T_0}(v) = 1 + |\phi^{-1}(v)|$ if $v = w_0$ or $v$ is a special vertex opposing $\phi(\Delta^*)$ for some $\Delta^* \in X(G)$, and $\deg_{T_0}(v) = 2 + |\phi^{-1}(v)|$ otherwise.
 This will be done by induction along the separation tree of $G$.
 
 If $G$ is just the triangle $\Delta_{\rm out} = w_0w_1w_2$ we set $T_0=\{w_0w_1,w_1w_2\}$, $T_1=\emptyset$, and $T_2=w_2w_0$, and we are done.
 
 Let now $\Delta= v_0v_1v_2 \in X(G)$ be an inclusion-minimal filled triangle of $G$.
 Hence $G_\Delta$ is $4$-connected and we have $\outer{G_\Delta} = G_\Delta$.
 We apply induction to the graph $G'$ obtained from $G$ by removing all vertices inside $\Delta$, together with the $k$-assignment $\phi'$ obtained from $\phi$ by restricting to $X(G') = X(G) \setminus \{\Delta\}$, and obtain three trees $T'_0,T'_1,T'_2$ with the desired properties. We have to include the edges of $G_\Delta - E(\Delta)$ into our decomposition. 
  
 
 Note that at most two vertices of $\Delta$ also appear in $\Delta_{\rm out}$ but by the choice of $w_0w_2$ these cannot be $w_0$ and $w_2$ simultaneously.
 Without loss of generality let $\phi(\Delta) = v_0$ and $v_2$ be such that $\{v_0,v_2\}\not\subseteq\{w_0,w_1,w_2\}$.
 Since $w_0w_2$ is in no separating triangle, we may also assume that $v_1 \notin \{w_0,w_2\}$.
 Let $F_0, F_1, F_2$ be a partition of the edges in $G_\Delta - E(\Delta)$ as given by Lemma~\ref{lem:inner-decomposition} for $x=0$, $y=1$ and $z=2$.
 (Recall that $\outer{G_\Delta} = G_\Delta$.)
 We include $F_0$ into $T'_0$, $F_1$ into $T'_1$, and $F_2$ into $T'_2$.
 Clearly, by~\ref{enum:parts-form-partition} we obtain a set of trees $T_0,T_1,T_2$ with the desired properties, except that it remains to bound the degrees of vertices in $T_0$. 

 First consider any inner vertex $v$ of $G_\Delta$.
 If $v = u_0$ is the vertex in $G_\Delta$ opposing $v_0 = \phi(\Delta)$, then $v$ has one incident edge in $F_0$ and no incident edge in $T'_0$, and thus $\deg_{T_0}(v) = 1 + |\phi^{-1}(v)|$ with $|\phi^{-1}(v)| = 0$.
 If $v$ is an inner vertex of $G_\Delta$ different from $u_0$, then $v$ has two incident edges in $F_0$ and no incident edge in $T'_0$, and thus $\deg_{T_0}(v) = 2 + |\phi^{-1}(v)|$ with $|\phi^{-1}(v)| = 0$.
 
 Now consider any vertex $v$ of $G'$.
 Note that $v$ is a special vertex opposing $\phi(\Delta^*)$ for some $\Delta^* \in X(G)$ if and only if that is also the case for $\phi'$ and $G'$.
 If $v = v_0$, then $|\phi^{-1}(v)| = |\phi'^{-1}(v)| + 1$ and $\deg_{T_0}(v) = \deg_{T'_0}(v) + 1$, as $v$ has exactly one incident edge in $F_0$, which implies the desired degree of $v$ in $T_0$ for this case.
 Finally, if $v$ is any vertex of $G'$ different from $v_0$, then $|\phi^{-1}(v)| = |\phi'^{-1}(v)|$ and $\deg_{T_0}(v) = \deg_{T'_0}(v)$, as $v$ has no incident edge in $F_0$, which implies the desired degree of $v$ in $T_0$ also for this case.
\end{proof}

\section{Proofs of Decomposition Results}\label{sec:proofs}

In this section, we show that all the classes of concern admit triangle assignments qualifying for Proposition~\ref{prop:decomposition-from-assignment} with respect to the claimed parameters. This will suffice to prove the Theorems of this paper.

\begin{figure}[tb]
 \centering
 \includegraphics{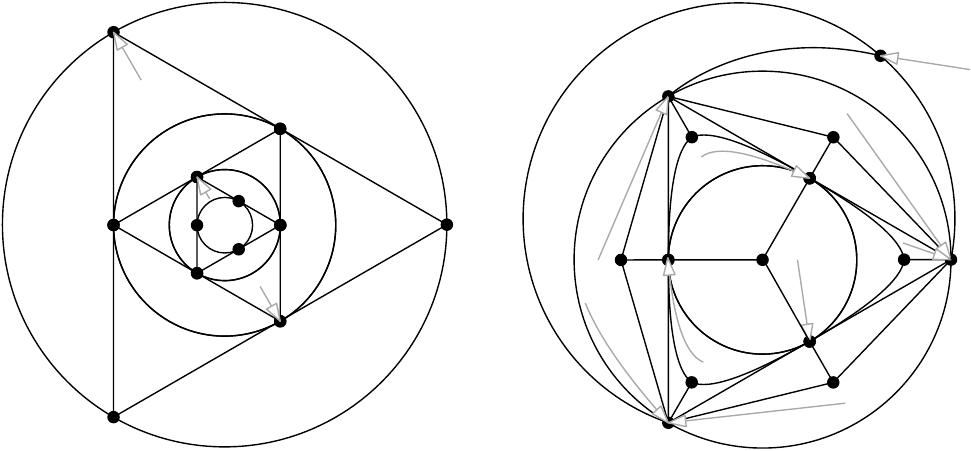}
 \caption{A non-$4$-connected planar triangulation with $0$-assignment qualifying for Proposition~\ref{prop:decomposition-from-assignment} and a non-Hamiltonian planar triangulation with $1$-assignment qualifying for Proposition~\ref{prop:decomposition-from-assignment}.}
 \label{fig:christmas-tree}
\end{figure}

Now Proposition~\ref{prop:decomposition-from-assignment} immediately implies Theorem~\ref{thm:2t1p}, namely that every $4$-connected triangulation is $[2,2]^\star$-decomposable.

\begin{proof}[Proof of Theorem~\ref{thm:2t1p}]
 Clearly, if $G$ is a $4$-connected plane triangulation, then $X(G)$ consists only of the outer face $\Delta_{\rm out}=xyz$ and $\phi(\Delta_{\rm out}) = x$ is a $0$-assignment of $G$.
 Since $G$ contains no separating triangles, applying Proposition~\ref{prop:decomposition-from-assignment} yields the result.
\end{proof}

We remark that some plane triangulations (such as the one in the left of Figure~\ref{fig:christmas-tree}) are not $4$-connected and still admit a $0$-assignment qualifying for Proposition~\ref{prop:decomposition-from-assignment}.

Next, we shall turn our attention to general plane triangulations.
Recall that $Y(G)$ denotes the set of all inner faces of a plane triangulation $G$.

\begin{lem}\label{lem:inner-face-assignment}
 Let $G$ be a plane triangulation on at least four vertices and $u$ be a special vertex of $G$.
 Then there is a map $\psi : Y(G) \to V(G)$ with $\psi(\Delta) \in V(\Delta)$ for every $\Delta \in Y(G)$ and
 \[
  |\psi^{-1}(v)| = 
   \begin{cases}
    0 & \text{ if $v$ is an outer vertex,}\\
    2 & \text{ if $v$ is an inner vertex, $v \neq u$,}\\
    3 & \text{ if $v = u$.}
   \end{cases}
 \]
\end{lem}
\begin{proof}
 Let $v_1,v_2$ be two outer vertices adjacent to the special vertex $u$ and $\Delta^\star \in Y(G)$ be the inner face formed by $v_1$, $v_2$ and $u$.
 Consider a plane straight-line embedding of $G$ in which the only horizontal edge is $v_1v_2$.
 For every inner face $\Delta \in Y(G) - \Delta^\star$, let $\psi(\Delta)$ be the vertex in $V(\Delta)$ with the middle $y$-coordinate.
 Moreover, let $\psi(\Delta^\star) = u$.
 It is easily seen that $\psi$ has the desired properties.
\end{proof}

\begin{lem}\label{lem:2-assignment}
 Every plane triangulation admits a $2$-assignment.
\end{lem}
\begin{proof}
 Let $G$ be any plane triangulation with outer triangle $\Delta_{\rm out}$.
 By~\ref{enum:separating-triangles-are-faces} we have that $\{\Delta_{\rm out}\} \cup \bigcup_{\Delta \in X(G)} Y(\outer{G_\Delta}) \supset X(G)$.
 Similarly, every inner vertex of $G$ is an inner vertex of $\outer{G_\Delta}$ for exactly one $\Delta \in X(G)$.
 By Lemma~\ref{lem:inner-face-assignment} we have maps $\psi_\Delta : Y(\outer{G_\Delta}) \to V(\outer{G_\Delta})$ mapping the inner faces of $\outer{G_\Delta}$ to inner vertices of $\outer{G_\Delta}$ such that every inner vertex is hit twice, except for a fixed special vertex of $\outer{G_\Delta}$, which is hit three times.
 From this collection of maps $\{\psi_\Delta \mid \Delta \in X(G)\}$ we obtain a desired $2$-assignment $\phi : X(G) \to V(G)$ by setting $\phi(\Delta_{\rm out}) = v$ for some outer vertex $v$ of $G$ and $\phi(\Delta') = \psi_\Delta(\Delta')$ whenever $\Delta' \in Y(G_\Delta)$.
\end{proof}

We can now give an alternative proof of Gon{\c{c}}alves's result~\cite{Gon-09} that every planar graph is $(2,4)^\star$-decomposable and indeed a slight strengthening thereof, since in~\cite{Gon-09} the forests of the decomposition are not necessarily connected.
That is, we prove Theorem~\ref{thm:planar} stating that every planar triangulation decomposes into two trees and one spanning tree of maximum degree~$4$.

\begin{proof}[Proof of Theorem~\ref{thm:planar}]
 Lemma~\ref{lem:2-assignment} gives a $2$-assignment $\phi$ for a planar triangulation for any prescribed outer triangle $\Delta_{\rm out}$, where moreover the choice of $\phi(\Delta_{\rm out})$  is arbitrary. Since clearly any planar triangulation contains an edge $e$, that is not contained in a separating triangle we can choose an outer triangle containing $e$. Use Lemma~\ref{lem:2-assignment} to get a $2$-assignment such that $\phi(\Delta_{\rm out})$  is a vertex of $e$ and apply Proposition~\ref{prop:decomposition-from-assignment} to get the desired decomposition.
\end{proof}


Next, we shall turn our attention to Hamiltonian planar triangulations.


\begin{lem}\label{lem:1-assignment}
 Let $G$ be a plane triangulation with outer triangle $\Delta_{\rm out} = v_0v_1v_2$.
 \begin{enumerate}[label = (\roman*)]
  \item If $G$ admits a Hamiltonian $v_0$-$v_2$-path, then $G$ admits a $1$-assignment $\phi$ with $\phi(\Delta_{\rm out}) = v_1$ and 
  \[
   |\phi^{-1}(v)| \leq 
    \begin{cases}
     1 & \text{ if $v = v_1,v_2$,}\\
     0 & \text{ if $v = v_0$.}
    \end{cases}
  \]
  \label{item:assign-v_1-and-split}
  
  \item If $G$ admits a Hamiltonian $v_0$-$v_2$-path, then $G$ admits a $1$-assignment $\phi$ with $\phi(\Delta_{\rm out}) = v_1$ and 
  \[
   |\phi^{-1}(v)| \leq 
    \begin{cases}
     2 & \text{ if $v = v_1$,}\\
     0 & \text{ if $v = v_0,v_2$.}
    \end{cases}
  \]
  \label{item:assign-v_1-and-both}
  
  \item If $G - v_1$ admits a Hamiltonian $v_0$-$v_2$-path, then $G$ admits a $1$-assignment $\phi$ with $\phi(\Delta_{\rm out}) = v_2$ and 
  \[
   |\phi^{-1}(v)| \leq 
    \begin{cases}
     1 & \text{ if $v = v_2$,}\\
     0 & \text{ if $v = v_0,v_1$.}
    \end{cases}
  \]
  \label{item:assign-v_2}
 \end{enumerate}
\end{lem}
\begin{proof}
 Let $G$ be a plane triangulation with outer triangle $\Delta_{\rm out} = v_0v_1v_2$ and let $P$ be a fixed Hamiltonian $v_0$-$v_2$-path in $G$ for items~\ref{item:assign-v_1-and-split},\ref{item:assign-v_1-and-both} or in $G-v_1$ for item~\ref{item:assign-v_2}.
 We shall prove all three items by induction on $|X(G)|$, the number of filled triangles in $G$.
 If $|X(G)| = 1$, i.e., $G$ is $4$-connected, then a desired $1$-assignment $\phi$ is given by $\phi(\Delta_{\rm out}) = v_1$ for~\ref{item:assign-v_1-and-split},\ref{item:assign-v_1-and-both} and $\phi(\Delta_{\rm out}) = v_2$ for~\ref{item:assign-v_2}, respectively.
 
 So assume that $|X(G)| \geq 2$, i.e., $G$ has at least one separating triangle.
 Let $Z \subseteq X(G) - \Delta_{\rm out}$ denote the set of all inclusion-maximal separating triangles in $G$. 
 For each $\Delta \in Z$, let $P_\Delta \subset P$ denote the inclusion-maximal subpath of $P$ consisting only of edges in $G_\Delta - E(\Delta)$.
 Then $P_\Delta$ has distinct endpoints, which both lie on $\Delta$.
 Note that $P_\Delta$ is a Hamiltonian path either in $G_\Delta$, or in $G_\Delta - v$ for one $v \in \Delta$.
 Moreover, $E(P_\Delta) \subset E(G_\Delta) - E(\Delta)$ for each $\Delta \in Z$ and thus $E(P_\Delta) \cap E(P_{\Delta'}) = \emptyset$ for $\Delta \neq \Delta' \in Z$, by~\ref{enum:parts-form-partition}.
 
 Let $u$ denote the special vertex of $G$ opposing $v_1$ for items~\ref{item:assign-v_1-and-split},\ref{item:assign-v_1-and-both} and opposing $v_0$ for item~\ref{item:assign-v_2}.
 We endow the path $P$ in $G$ with an orientation by orienting every edge towards $u$.
 This implies an orientation of $P_\Delta$ for each $\Delta \in Z$.
 Now depending on $P_\Delta$ and its orientation, we consider a $1$-assignment $\phi_\Delta$ of $G_\Delta$ for each $\Delta \in Z$, which we obtain by induction.
 
 \begin{description}
  \item[Case 1: $u$ is an inner vertex of $P_\Delta$.]{\ \\}
   Denote by $v_0(\Delta)$ and $v_2(\Delta)$ be the endpoints of $P_\Delta$.
   Since $u$ is a special vertex of $G$, we have $\Delta = v_0(\Delta)uv_2(\Delta)$.
   (Recall that, by definition, special vertices are not contained in any separating triangle of $G$.)
   In particular $P_\Delta$ is a Hamiltonian path in $G_\Delta$.
   Then by induction (item~\ref{item:assign-v_1-and-both}) there exists a $1$-assignment $\phi_\Delta$ of $G_\Delta$ with $|\phi_\Delta^{-1}(u)| \leq 2$ and $|\phi_\Delta^{-1}(v)| = 0$ for $v = v_0(\Delta),v_2(\Delta)$.
 
  \item[Case 2: $P_\Delta$ is a Hamiltonian path in $G_\Delta$.]{\ \\}
   Let $v_0(\Delta)$ and $v_2(\Delta)$ be the endpoints of $P_\Delta$ so that $P_\Delta$ is oriented from $v_0(\Delta)$ to $v_2(\Delta)$ and let $v_1(\Delta)$ be the third outer vertex of $G_\Delta$.
   Then by induction (item~\ref{item:assign-v_1-and-split}) there exists a $1$-assignment $\phi_\Delta$ of $G_\Delta$ with $|\phi_\Delta^{-1}(v)| \leq 1$ for $v = v_1(\Delta),v_2(\Delta)$ and $|\phi_\Delta^{-1}(v_0(\Delta))| = 0$.
   
  \item[Case 3: $P_\Delta$ is a Hamiltonian path in $G_\Delta - v(\Delta)$.]{\ \\}
   Let $v_0(\Delta)$ and $v_2(\Delta)$ be the endpoints of $P_\Delta$ so that $P_\Delta$ is oriented from $v_0(\Delta)$ to $v_2(\Delta)$.
   Note that $\Delta = v_0(\Delta)v_1(\Delta)v_2(\Delta)$.
   Then by induction (item~\ref{item:assign-v_2}) there exists a $1$-assignment $\phi_\Delta$ of $G_\Delta$ with $|\phi_\Delta^{-1}(v_2(\Delta))| = 1$ and $|\phi_\Delta^{-1}(v)| = 0$ for $v = v_0(\Delta),v_1(\Delta)$.
 \end{description}

 Finally, define a map $\phi : X(G) \to V(G)$ by $\phi(\Delta_{\rm out}) = v_1$ for items~\ref{item:assign-v_1-and-split},\ref{item:assign-v_1-and-both} and $\phi(\Delta_{\rm out}) = v_2$ for item~\ref{item:assign-v_2} and $\phi(\Delta) = \phi_{\Delta'}(\Delta)$ for the unique $\Delta^\star \in Z$ with $\Delta \subset G_{\Delta^\star}$.
 It is straightforward to check that $\phi$ satisfies the desired requirements.
\end{proof}

Having Lemma~\ref{lem:1-assignment}, we can deduce from Proposition~\ref{prop:decomposition-from-assignment} a $[2,3]^\star$-decomposition of any Hamiltonian triangulation.
That is, we prove Theorem~\ref{thm:hamil} stating that every Hamiltonian planar triangulation decomposes into two trees and one spanning tree of maximum degree~$3$.

\begin{proof}[Proof of Theorem~\ref{thm:hamil}]
 Let $G$ be a Hamiltonian plane triangulation with Hamiltonian cycle $C$. Consider counterclockwise consecutive vertices $v_1, v_2, v_0$ on $C$, such that $v_2$ has no neighbors in the interior of $C$. The latter is possible because $C$ with the edges in its interior is a maximal outerplanar graph and thus has a degree $2$ vertex. Thus, $v_0v_1v_2$ is a facial triangle of $G$. We show that one edge of $v_0v_1v_2$ is in no separating triangle. Suppose, that $v_1v_0$ is in a separating triangle $\Delta$. Then an edge of $\Delta$ incident with $v_1$ or $v_0$ has to lie entirely in the exterior of $C$ - say $v_1$ is that vertex. Now, the edge $v_1v_2$ cannot lie in a separating triangle because all edges that could be used to form such a triangle lie in the exterior of $C$. In the other case $v_2v_0$ lies in no separating triangle.
 
 Now, we embed $G$ such that $v_0v_1v_2=\Delta_{\rm out}$ is the outer triangle of the embedding and after possibly renaming $v_1, v_2, v_0$ we have that $G$ admits a Hamiltonian $v_0$-$v_2$-path and the edge $v_0v_1$ is in no separating triangle. Now, Lemma~\ref{lem:1-assignment} gives that $G$ admits a $1$-assignment with $\phi(\Delta_{\rm out}) = v_1$. Since $v_0v_1$ is in no separating triangle Proposition~\ref{prop:decomposition-from-assignment} yields the desired decomposition.
\end{proof}

\section{Tightness}\label{sec:tight}
In the present section we show that all our results are best-possible. More precisely, any of the classes for that we show $[2,d]^\star$-decomposability contains members that are not even $(2,d-1)$-decomposable. Moreover, there are triangulations ``close to the class'' that are not $(2,d)$-decomposable.

Recall that for a graph $G$ and integer $d \geq 0$ we say that $G$ is $(2,d)$-decomposable (respectively $(2,d)^\star$-decomposable, $[2,d]^\star$-decomposable) if the edges of $G$ can be partitioned into two forests and a graph (respectively a forest, tree) of maximum degree~$d$. For this section we will furthermore say that $G$ is $[2,d]$-decomposable if the edges of $G$ can be partitioned into two forests and a connected graph of maximum degree~$d$.

Let $G$ be a planar triangulation.
We prove in Theorem~\ref{thm:planar} that $G$ is $[2,4]^\star$-decomposable, in Theorem~\ref{thm:hamil} that if $G$ is Hamiltonian, then $G$ is even $[2,3]^\star$-decomposable, and in Theorem~\ref{thm:2t1p} that if $G$ is $4$-connected, then $G$ is even $[2,2]^\star$-decomposable.
In this section we want to argue that Theorems~\ref{thm:2t1p} and~\ref{thm:hamil} are best-possible in some sense, namely that these cannot be extended to triangulations that are ``a decent amount away from'' $4$-connectivity, respectively Hamiltonicity.
To this end let $G' \subseteq G$ be a subgraph on at least four vertices of $G$ that is itself a triangulation.
Let $n \geq 4$ be the number of vertices in $G'$, and $k$ be the number of faces of $G'$ that are not faces of $G$, i.e., $k = |Y(G')-Y(G)|$.

Note that if $G$ is $4$-connected (and thus $[2,2]^\star$-decomposable by Theorem~\ref{thm:2t1p}), then necessarily $G' = G$ and thus $k = 0$.
On the other hand, if $G$ has a Hamiltonian cycle $C$ (and is thus $[2,3]^\star$-decomposable by Theorem~\ref{thm:hamil}), then along $C$ between the interior vertices of any two faces in $Y(G')-Y(G)$ there is at least one vertex of $G'$, showing that $k \leq n$.
So we have shown $[2,d]^\star$-decomposability for $d=2,3$ in cases where $k$ is relatively small.
Next we show that indeed $[2,d]^\star$- and $(2,d)^\star$-decomposability for $d=2,3$ can be achieved \emph{only if} $k$ is relatively small.

As we also want to make a slightly stronger statement involving $[2,d]$- and $(2,d)$-decomposability for $d=2,3$, we introduce the following notation.
For fixed triangulations $G' \subseteq G$, we say that a $[2,d]$- or $(2,d)$-decomposition $(F_1,F_2,H_3)$ is \emph{special}, if for every face $\Delta \in Y(G')-Y(G)$ we have that the subgraph of $H_3$ consisting of all edges in the interior of $\Delta$ but not on the boundary of $\Delta$ is a forest.
Note in particular that every $(2,d)^\star$-decomposition and every $[2,d]^\star$-decomposition is special.

\begin{prop}\label{prop:tightness}
 Let $G' \subseteq G$ be two triangulations, $n$ be the number of vertices in $G'$, and $k$ be the number of faces of $G'$ that are not faces of $G$, i.e., $k = |Y(G')-Y(G)|$.
 Then each of the following holds:
 \begin{enumerate}[label=(\roman*)]
  \item If $k \geq 6$, then $G$ has no special $[2,2]$-decomposition.\label{enum:[2,2]}
  \item If $k \geq 9$, then $G$ has no special $(2,2)$-decomposition.\label{enum:(2,2)}
  \item If $k \geq n+6$, then $G$ has no special $[2,3]$-decomposition.\label{enum:[2,3]}
  \item If $k \geq n+9$, then $G$ has no special $(2,3)$-decomposition.\label{enum:(2,3)}
 \end{enumerate}
\end{prop}
\begin{proof}
 We shall prove all four items by contraposition, i.e., we assume that $G$ has a special $(2,d)$- or $[2,d]$-decomposition for $d \in \{2,3\}$ and derive from this an upper bound on $k$.
 So let $(F_1,F_2,H_3)$ be a special decomposition of $E(G)$ into two forests $F_1,F_2$ and a third graph $H_3$ of maximum degree~$d$.
 As we shall count components, we consider each of $F_1,F_2,H_3$ to be a spanning subgraph of $G$, i.e., to contain all the vertices of $G$.
  
 \begin{claim}\label{claim:tree-components}
  Let $J$ be a graph with $n$ vertices and $kn-x$ edges.
  If the edges of $J$ are partitioned into $k$ spanning subgraphs whose components are only trees and cycles, then the total number of tree components in these $k$ graphs is equal to $x$.
 \end{claim}

 Indeed, it is enough to observe that a spanning subgraph $J'$ of $J$ whose components are trees and cycles contains exactly $n-c$ edges of $G$, where $c$ is the number of tree components of $J'$.
 
 \medskip

 For a face $\Delta$ in $Y(G') - Y(G)$, let $G[\Delta]$ denote the maximal subgraph of $G$ whose vertices are on or inside $\Delta$ and whose edges have at least one endpoint inside $\Delta$. In other words $G[\Delta]=G_{\Delta}-E(\Delta)$.
 
 \begin{claim}\label{claim:in-Delta}
  For every $\Delta \in Y(G') - Y(G)$ at least one of the following holds:
  \begin{itemize}
   \item Two vertices of $\Delta$ are in the same component of $F_1|_{G[\Delta]}$.
   \item Two vertices of $\Delta$ are in the same component of $F_2|_{G[\Delta]}$.
   \item A vertex of $\Delta$ is incident to an edge in $H_3|_{G[\Delta]}$.
  \end{itemize}
 \end{claim} 
 
 Indeed, if $G[\Delta]$ has $n$ vertices, then it has $3n-9$ edges.
 As the edges of $G[\Delta]$ are partitioned into three forests $F_1|_{G[\Delta]},F_2|_{G[\Delta]},H_3|_{G[\Delta]}$ (the latter graph is a forest as the decomposition is special), by Claim~\ref{claim:tree-components} there are in total exactly $9$ components.
 Now if the first item does not hold, then $F_1|_{G[\Delta]}$ has at least $3$ components.
 Similarly, if the second items does not hold, then $F_2|_{G[\Delta]}$ has at least $3$ components.
 Finally, if the third item does not hold, then $H_3|_{G[\Delta]}$ has at least $4$ components.
 Thus if none of the items hold, then there are at least $10$ components in total, which is a contradiction and proves the claim.
 
 \medskip

 Taking the restrictions of $F_1,F_2,H_3$ to the subgraph $G'$ of $G$, we obtain corresponding forests $F'_1,F'_2$ and graph $H'_3$ partitioning $E(G')$.
 For a subgraph $J \in \{F_1,F_2,H_3,F'_1,F'_2,H'_3\}$ of $G$ let $c(J)$ denote the number of its tree components that contain a vertex of $G'$.
 The crucial insight is that the number of faces $\Delta \in Y(G')-Y(G)$ for which the first, respectively second, item of Claim~\ref{claim:in-Delta} holds, is at most $c(F'_1)-c(F_1)$, respectively $c(F'_2)-c(F_2)$.
 Every further face $\Delta \in Y(G')-Y(G)$ contributes at least one edge in $E(H_3) - E(H'_3)$ incident to a vertex of $G'$.
  
 \begin{description}
  \item[Case~1: $d=2$.]{\ \\}
   Claim~\ref{claim:tree-components} applied to the partition $(F'_1,F'_2,H'_3)$ of $G'$ gives
   \[
    c(F'_1) + c(F'_2) + c(H'_3) = 6.
   \]
   Since $G' \subseteq G$, we have $c(F'_1) \geq c(F_1) \geq 1$, $c(F'_2) \geq c(F_2) \geq 1$, and $c(H'_3) \geq c(H_3)$, which implies that 
   \[
    c(H_3) \leq c(H'_3) = 6 - c(F'_1) - c(F'_2) \leq 4.
   \]
   For at most $c(F'_1)- c(F_1)$, respectively $c(F'_2) - c(F_2)$ faces $\Delta \in Y(G')-Y(G)$ the first, respectively second, item of Claim~\ref{claim:in-Delta} holds.
   For at most $c(H'_3) - c(H_3)$ faces $\Delta \in Y(G')-Y(G)$ two vertices of $\Delta$ are in the same component of $H_3|_{G[\Delta]}$.
   As $H_3$ has maximum degree at most~$d=2$, for at most $2c(H_3)$ further faces $\Delta \in Y(G')-Y(G)$ the last item of Claim~\ref{claim:in-Delta} holds.
   Hence, the total number of faces in $Y(G') - Y(G)$ is at most 
   \begin{eqnarray*}
    \quad c(F'_1) - c(F_1) + c(F'_2) - c(F_2) + c(H'_3) - c(H_3) + 2c(H_3) \hspace{4em} \\
     \leq 6 - 2 + c(H_3) = 4 + c(H_3).
   \end{eqnarray*}
   Thus $k \leq 4+c(H_3) = 5$ for $c(H_3)=1$ which proves~\ref{enum:[2,2]} and $k \leq 4+c(H_3) \leq 8$ otherwise, which proves~\ref{enum:(2,2)}.
   
  \item[Case~2: $d=3$.]{\ \\}
   As $H'_3$ has maximum degree at most~$d=3$, we have $|E(H'_3)| = 3n/2 - y$ for some $y$. 
   Claim~\ref{claim:tree-components} applied to the partition $(F'_1,F'_2)$ of $G' - E(H'_3)$ gives $c(F'_1) + c(F'_2) = n/2 + 6 - y$ and thus
   \[
    y = n/2 + 6 - c(F'_1) - c(F'_2) \leq n/2 + 4.
   \]
   Moreover, $c(H'_3) \geq \max\{1,n-|E(H'_3)|\}$ and thus
   \[
    y = \frac{3}{2}n - |E(H'_3)| \leq n/2 + c(H'_3).
   \]
   Clearly, $c(F_1), c(F_2) \geq 1$, and thus for at most $c(F'_1)- c(F_1) + c(F'_2) - c(F_2) \leq n/2 + 4 - y$ faces $\Delta \in Y(G')-Y(G)$ the first or second item of Claim~\ref{claim:in-Delta} holds.
   
   Now consider the set $S = \{(v,e) \mid v \in V(G'), e \in E(H_3), v \in e\}$.
   As $H_3$ has maximum degree at most~$d=3$, we have $|S| \leq 3n$.
   There are exactly $2|E(H'_3)| = 3n-2y$ pairs $(v,e) \in S$ with $e \in E(H'_3)$.
   For at most $c(H'_3) - c(H_3)$ faces $\Delta \in Y(G')-Y(G)$ two vertices of $\Delta$ are in the same component of $H_3|_{G[\Delta]}$, which corresponds to at least $2(c(H'_3)-c(H_3))$ further pairs $(v,e) \in S$ with $e \in E(G[\Delta])$ for such faces $\Delta$.
   Every further face $\Delta \in Y(G')-Y(G)$ for which the last item of Claim~\ref{claim:in-Delta} holds corresponds to at least one further pair $(v,e) \in S$ with $e \in E(G[\Delta])$, giving that there are at most $|S| - (3n-2y) - 2(c(H'_3)-c(H_3)) = 2(y - c(H'_3) + c(H_3))$ such faces $\Delta \in Y(G')-Y(G)$.
   Hence, for the total number $k$ of faces in $Y(G') - Y(G)$ we have 
   \begin{align*}
    k & \leq (n/2 + 4 - y) + (c(H'_3) - c(H_3)) + 2(y - c(H'_3) + c(H_3)) \\
    & = n/2 + 4 + y + (c(H_3) - c(H'_3)) \\
    & \leq n + 4 + \min\{c(H_3),4\}.
   \end{align*}   
   Thus $k \leq n+4+\min\{c(H_3),4\} = n+5$ for $c(H_3)=1$ which proves~\ref{enum:[2,3]} and $k \leq n+4+\min\{c(H_3),4\} \leq n+8$ otherwise, which proves~\ref{enum:(2,3)}.\qedhere
 \end{description}
\end{proof}

Let us mention that Proposition~\ref{prop:tightness} is tight in the following sense:
Figure~\ref{fig:non-tight-examples} shows triangulations $G' \subset G$ with $k = |Y(G')-Y(G)|$ being one less than the bound in Proposition~\ref{prop:tightness} and yet $G$ has a special $[2,d]$- or $(2,d)$-decomposition (even stronger: $[2,d]^\star$-, respectively $(2,d)^\star$-decomposition) for $d \in \{2,3\}$ as in the respective case in the theorem.

\begin{figure}[tb]
 \centering
 \includegraphics{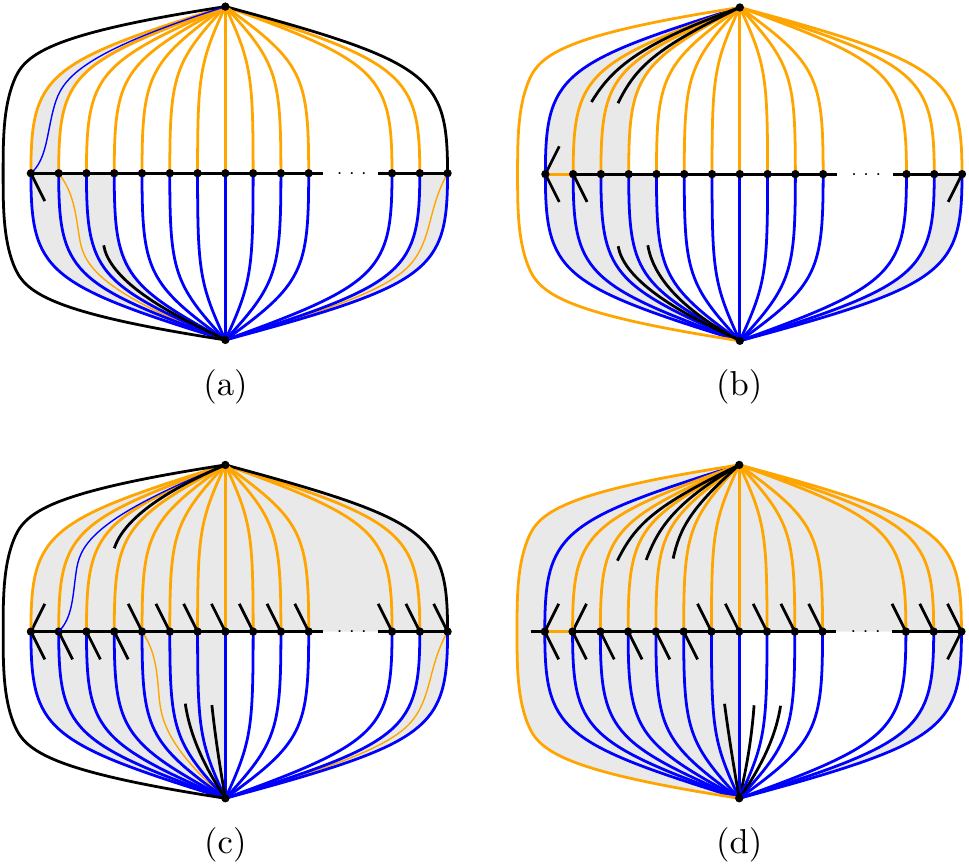}
 \caption{Examples showing that the bounds on $k = |Y(G')-Y(G)|$ in Proposition~\ref{prop:tightness} are best-possible.
  The triangulation $G'\subset G$ is drawn with thick edges, gray areas indicate faces in $Y(G')-Y(G)$, and edge colors indicate a claimed decomposition.
  (a) $k = 5$ and $G$ is $[2,2]^\star$-decomposable.
  (b) $k = 8$ and $G$ is $(2,2)^\star$-decomposable.
  (c) $k = n + 5$ and $G$ is $[2,3]^\star$-decomposable.
  (d) $k = n + 8$ and $G$ is $(2,3)^\star$-decomposable.}
 \label{fig:non-tight-examples}
\end{figure}

Finally, Proposition~\ref{prop:tightness} implies Theorem~\ref{thm:best-possible} stating that \ref{enum:4-connected-not-2,1} some $4$-connected triangulations are not $(2,1)$-decomposable, \ref{enum:Hamiltonian-not-2,2} some Hamiltonian triangulations are not $(2,2)$-decomposable, and \ref{enum:planar-not-2,3} some planar triangulations are not $(2,3)$-decomposable.

\begin{proof}[Proof of Theorem~\ref{thm:best-possible}]
 The first item is a simple counting argument.
 
 \begin{enumerate}[label = (\roman*)]
  \item Let $G_1$ be any $4$-connected triangulation on $n \geq 9$ vertices.
   Note that every $n$-vertex $(2,1)$-decomposable graph has at most $2(n-1) + n/2$ edges, since it decomposes into two forests and a matching.
   On the other hand, $G_1$ has $3n-6$ edges, which is strictly more than $2(n-1)+n/2$ for $n \geq 9$, and thus $G_1$ is not $(2,1)$-decomposable.
 \end{enumerate}
 
 The key observation for~\ref{enum:Hamiltonian-not-2,2} and~\ref{enum:planar-not-2,3} is that if $G' \subseteq G$ are triangulations and every face $\Delta \in Y(G') - Y(G)$ contains exactly one vertex of $V(G) - V(G')$, then every $(2,d)$-decomposition of $G$ is special.
 Hence we can use Proposition~\ref{prop:tightness} to argue that such $G$ admits no $(2,d)$-decomposition for $d \in \{2,3\}$, provided $k = |Y(G')-Y(G)|$ is big enough.
 
 \begin{enumerate}[label = (\roman*), start = 2]
  \item Let $G'_2$ be any Hamiltonian triangulation on an even number $n \geq 18$ vertices and let $C$ be a Hamiltonian cycle in $G'_2$.
   Now for every other edge $e$ in $C$ add a new vertex $v_e$ in one of the faces incident to $e$, making $v_e$ adjacent to all vertices of this face.
   As we picked every other edge, we have $v_e \neq v_{e'}$ for $e \neq e'$.
   The resulting graph $G_2$ is a triangulation satisfying $G'_2 \subset G_2$ and $k =  n/2  \geq 9$.
   Hence, by Proposition~\ref{prop:tightness}\ref{enum:(2,2)} $G_2$ has no special $(2,2)$-decomposition, and as every $\Delta \in Y(G'_2) - Y(G_2)$ contains exactly one vertex of $V(G_2) - V(G'_2)$, $G_2$ has no $(2,2)$-decomposition at all.
   Finally, $G_2$ is Hamiltonian as the cycle $C$ can be easily rerouted to also contain every vertex $v_e \in V(G_2) - V(G'_2)$. 
   
  \item Let $G'_3$ be any triangulation on $n \geq 12$ vertices.
   Let $G_3$ be the triangulation arising from $G'_3$ by adding a new vertex in each face, making it adjacent to all vertices of this face.
   Note that $G'_3 \subset G_3$ and $k = |Y(G'_3)-Y(G_3)| = 2n-3 \geq n+12-3 = n+9$.
   Hence, by Proposition~\ref{prop:tightness}\ref{enum:(2,3)} $G_3$ has no special $(2,3)$-decomposition, and as every $\Delta \in Y(G'_3) - Y(G_3)$ contains exactly one vertex of $V(G_3) - V(G'_3)$, $G_3$ has no $(2,3)$-decomposition at all.
 \end{enumerate}
\end{proof}

\section{Conclusions}\label{sec:conc}

Gon\c{c}alves~\cite{Gon-09} showed that every planar graph admits a $(2,4)^\star$-decomposition.
In this paper we showed that structural properties of planar triangulations allow for $[2,d]^\star$-decompositions with $d < 4$.
Moreover, note that our results are slightly stronger than just showing decomposability into forests, since for planar triangulations we give decompositions into trees and the bounded degree tree is spanning. 
In light of these results and the Nine Dragon Tree Theorem one can ask under whhat conditions a graph is coverable by $k$ trees and a bounded degree $d$ tree or connected graph.

\subsection*{Acknowledgments}

We would like to thank the anonymous referees for carefully checking an earlier version of this paper and their comments, which have significantly improved the presentation of our results.

\bibliographystyle{abbrv}
\bibliography{lit}

\end{document}